\documentclass[a4paper,10pt]{article}

\usepackage[english]{babel}
\usepackage{amsmath,amsfonts,amssymb,amsthm}
\usepackage{bbm}
\usepackage{indentfirst}
\usepackage{url}

\newtheorem{theorem}{Theorem}[section]

\newtheorem{lemma}{Lemma}[section]
\newtheorem{proposition}{Proposition}[section]

\theoremstyle{definition}
\newtheorem{remark}{Remark}[section]

\numberwithin{equation}{section}

\newcommand\blfootnote[1]{\begingroup\renewcommand\thefootnote{}\footnote{#1}\addtocounter{footnote}{-1}\endgroup}

\begin{document}

\title{
{\bf\Large
Positive subharmonic solutions \\ to nonlinear ODEs with indefinite weight}}

\author{
\vspace{1mm}
\\
{\bf\large Alberto Boscaggin}
\vspace{1mm}\\
{\it\small Department of Mathematics, University of Torino}\\
{\it\small via Carlo Alberto 10}, {\it\small 10123 Torino, Italy}\\
{\it\small e-mail: alberto.boscaggin@unito.it}\vspace{1mm}\\
\vspace{1mm}\\
{\bf\large Guglielmo Feltrin}
\vspace{1mm}\\
{\it\small SISSA - International School for Advanced Studies}\\
{\it\small via Bonomea 265}, {\it\small 34136 Trieste, Italy}\\
{\it\small e-mail: guglielmo.feltrin@sissa.it}\vspace{1mm}}

\date{}

\maketitle

\vspace{-2mm}

\begin{abstract}
\noindent
We prove that the superlinear indefinite equation
\begin{equation*}
u'' + a(t)u^{p} = 0,
\end{equation*}
where $p > 1$ and $a(t)$ is a $T$-periodic sign-changing function satisfying the (sharp) mean value condition
$\int_{0}^{T} a(t)~\!dt < 0$, has positive subharmonic solutions of order $k$ for any large integer $k$,
thus providing a further contribution to a problem raised by G.~J.~Butler in its pioneering paper \cite{Bu-76}. 
The proof, which applies to a larger class of indefinite equations, combines coincidence degree theory (yielding a positive harmonic solution) 
with the Poincar\'{e}-Birkhoff fixed point theorem (giving subharmonic solutions oscillating around it).
\end{abstract}

\section{Introduction}\label{section-1}

In this paper, we study the existence of \textit{positive subharmonic solutions} for nonlinear second order
ODEs with indefinite weight. To describe our results, throughout the introduction we focus our attention to the superlinear indefinite
equation
\begin{equation}\label{eqmod}
u'' + a(t)u^{p} = 0,
\end{equation} 
with $a(t)$ a sign-changing $T$-periodic function and $p > 1$, which has been indeed the main motivation for our investigation. 

Boundary value problems associated with sign-indefinite equations are quite popular in the qualitative theory of nonlinear ODEs. The existence of
\textit{oscillatory} periodic solutions to superlinear indefinite equations like
\begin{equation}\label{eqnodal}
u'' + a(t)g(u) = 0,
\end{equation}
where $g(u) \sim |u|^{p-1} u$ with $p > 1$, 
was first investigated by Butler in the pioneering paper \cite{Bu-76}. Later on, along this line of research, several contributions followed
(cf.~\cite{CaDaPa-02,PaZa-02,PaZa-04,TeVe-00}) and it is nowadays well known that equation \eqref{eqnodal} possesses infinitely many periodic solutions (both harmonic and subharmonic)
with a large number of zeros in the intervals where the weight function is positive, as well as globally defined solutions with chaotic-like oscillatory behavior.

On the other hand, the existence of \textit{positive} periodic solutions to equations like \eqref{eqmod}, even if already raised by Butler as an open problem in \cite[p.~477]{Bu-76},
was investigated only more recently.

In this regard, a first crucial observation is that a mean value condition on $a(t)$ turns out to be necessary for the existence of
positive $kT$-periodic solutions (with $k \geq 1$ an integer number); indeed, dividing equation \eqref{eqmod} by $u(t)^{p}$ and  integrating on $\mathopen{[}0,kT\mathclose{]}$, one readily obtains
\begin{equation*}
\int_{0}^{kT} a(t)~\!dt = - p \int_{0}^{kT} \biggl{(}\dfrac{u'(t)}{u(t)^{p}}\biggr{)}^{2} u(t)^{p-1}~\!dt,
\end{equation*}
so that (recalling that $a(t)$ is $T$-periodic)  
\begin{equation}\label{meancon}
\int_{0}^{T} a(t)~\!dt < 0.
\end{equation}
This fact was first observed by Bandle, Pozio and Tesei in \cite{BaPoTe-88}, showing that the condition $\int_\Omega a(x)~\!dx < 0$
is actually necessary and sufficient for the existence of a positive solution to the Neumann problem associated with the elliptic partial differential equation 
\begin{equation*}
\Delta u + a(x) u^{p} = 0, \qquad u \in \Omega \subseteq \mathbb{R}^{N},
\end{equation*}
in the sublinear case $0 < p < 1$ (notice, indeed, that the computation leading to \eqref{meancon} is valid for any $p > 0$,
both for periodic and Neumann boundary conditions, and possibly in a PDE setting). A similar result was then proved in the superlinear case $p>1$ in \cite{AlTa-93,BeCaDoNi-95},
using arguments from critical point theory.

To the best of our knowledge, periodic boundary conditions were explicitly taken into account only in the very recent paper \cite{FeZa-15ade}.
Therein, a topological approach based on Mawhin's coincidence degree was introduced to prove that the mean value condition \eqref{meancon} guarantees
the existence of a positive $T$-periodic solution for a large class of indefinite equations including \eqref{eqmod}. In such a way, a first affirmative answer to Butler's question can be given.

On one hand, this result seems to be optimal (in the sense that no more than one $T$-periodic solution can be expected for a general weight function with negative mean value);
on the other hand, however, it is known that positive solutions to 
\eqref{eqmod} can exhibit complex behavior for special choices of the weight function $a(t)$. More precisely, it was shown in \cite{BaBoVe-15,FeZa-pp2015}
(on the lines of previous results dealing with the Dirichlet and Neumann problems \cite{BoGoHa-05,Bo-11jmaa,FeZa-15jde,GaHaZa-03mod,GiGo-09}) that,
whenever $a(t)$ has large negative part (that is, $a(t) = q^{+}(t) - \mu q^{-}(t)$ with $\mu \gg 0$),
equation \eqref{eqmod} has infinitely many positive subharmonic solutions, as well as globally defined positive solutions with chaotic-like multibumb behavior.

It appears therefore a quite natural question if the sharp mean value condition \eqref{meancon} - besides implying the existence of a positive $T$-periodic solution
to \eqref{eqmod} - also guarantees the existence of positive subharmonic solutions.
Quite unexpectedly, as a corollary of our main results, we are able to show that the answer is always affirmative, thus providing a further contribution to Butler's problem.

\begin{theorem}\label{th-intro}
Assume that $a \colon \mathbb{R} \to \mathbb{R}$ is a sign-changing continuous and $T$-periodic function, having a finite number of zeros in $\mathopen{[}0,T\mathclose{[}$ and
satisfying the mean value condition \eqref{meancon}.
Then, equation \eqref{eqmod} has a positive $T$-periodic solution, as well as positive subharmonic solutions of order $k$, for any large integer number $k$. 
\end{theorem} 

Actually, the assumptions on the weight function $a(t)$ can be considerably weakened and the conclusion about the number of subharmonic solutions obtained
can be made much more precise. We refer to Section~\ref{section-3} for more general statements.

Let us emphasize that investigating the existence of subharmonic solutions for time-periodic ODEs is often a quite delicate issue,
the more difficult point being the proof of the minimality of the period. In Theorem~\ref{th-intro}, $kT$-periodic solutions $u_{k}(t)$ are found 
(for $k$ large enough) oscillating
around a $T$-periodic solution $u^{*}(t)$ and a precise information on the number of zeros of $u_{k}(t) - u^{*}(t)$ is the key point in showing that $kT$ is the minimal period of $u_{k}(t)$.
This approach, based on the celebrated Poincar\'{e}-Birkhoff fixed point theorem, was introduced (and then applied to Ambrosetti-Prodi type periodic problems)
in the paper \cite{BoZa-13}, to which we also refer for a quite complete bibliography about the theme of subharmonic solutions.
It is worth noticing, however, that the application to equation \eqref{eqmod} of the method described in \cite{BoZa-13} is not straightforward.
First, due to the superlinear character of the nonlinearity, we cannot guarantee (as needed for the application of dynamical systems techniques) the global continuability of
solutions to \eqref{eqmod} (see \cite{BuGr-71}) and some careful a-priori bounds have to be performed.
Second, due to the indefinite character of the equation, it seems impossible to perform explicit estimates on the solutions in order to prove
the needed twist-condition of the Poincar\'{e}-Birkhoff theorem. To overcome this difficulty, we first use an idea from \cite{BoOrZa-14}
to develop an abstract variant of the main result in \cite{BoZa-13}, replacing an explicit estimate on the positive $T$-periodic solution $u^{*}(t)$
with an information about its Morse index. Using a clever trick by Brown and Hess (cf.~\cite{BrHe-90}), such an information is then easily achieved.
We emphasize this simple property here, since it is the crucial point for our arguments: \textit{any positive $T$-periodic solution of \eqref{eqmod} has non-zero Morse index}.

Let us finally recall that variational methods can be an alternative tool for the study of subharmonic solutions.
In this case, information about the minimality of the period can be often achieved with careful level estimates (see, among others, \cite{FoRa-94,SeTaTe-00}).
Maybe this technique can be successfully applied also to the superlinear indefinite equation \eqref{eqmod};
however, it has to be noticed that usually results obtained via a symplectic approach (namely, using the Poincar\'{e}-Birkhoff theorem)
give sharper information (see \cite{BoOrZa-14,BoZa-13}).

\medskip

The plan of the paper is the following. In Section~\ref{section-2} we present, on the lines of \cite{BoOrZa-14,BoZa-13}, an auxiliary result ensuring,
for a quite broad class of nonlinearities, the existence of subharmonic solutions oscillating around a $T$-periodic solution with non-zero Morse index.
In Section~\ref{section-3} we state our main results, dealing with equations of the type $u'' + a(t)g(u) = 0$ with $a(t)$ satisfying \eqref{meancon} and $g(u)$
defined on a (possibly bounded) interval of the type $\mathopen{[}0,d\mathclose{[}$; roughly speaking, we have that the existence of positive subharmonic solutions
(oscillating around a positive $T$-periodic solution) is always guaranteed whenever $g(u)$ is superlinear at zero and strictly convex, with $g(u)/u$ large enough near $u = d$.
Applications are given to equations superlinear at infinity (thus generalizing Theorem~\ref{th-intro}), to equations with a singularity as well as
to parameter-dependent equations.
In Section~\ref{section-4} we give the proof of these results; in more detail, we first prove (using a degree approach similarly as in \cite{FeZa-15ade},
together with the trick in \cite{BrHe-90}) the existence of a positive $T$-periodic solution with non-zero Morse index and we then apply the results of
Section~\ref{section-2} to obtain the desired positive subharmonic solutions around it.
Section~\ref{section-5} is devoted to some conclusive comments about our investigation.

\section{Morse index, Poincar\'{e}-Birkhoff theorem, subharmonics}\label{section-2}

In this section, we present our auxiliary result for the search of subharmonic solutions to scalar second order ODEs of the type
\begin{equation}\label{eqsub}
u'' + h(t,u) = 0,
\end{equation}
where $h \colon \mathbb{R}\times \mathbb{R} \to\mathbb{R}$ is a function $T$-periodic in the first variable (for some $T>0$).
Motivated by the applications to equations like \eqref{eqmod} with $a \in L^{1}(\mathopen{[}0,T\mathclose{]})$, we set up our result in a Carath\'{e}odory setting. More precisely,  
we assume that the function $h(t,u)$ is measurable in the $t$-variable, continuously differentiable in the $u$-variable
and satisfies the following condition: for any $r > 0$, there exists $m_{r} \in L^{1}(\mathopen{[}0,T\mathclose{]})$ such that
$\vert h(t,u) \vert + \vert \partial_u h(t,u) \vert \leq m_{r}(t)$ for a.e~$t\in\mathopen{[}0,T\mathclose{]}$
and for every $u\in\mathbb{R}$ with $\vert u \vert \leq r$. Of course, in view of this assumption, solutions to \eqref{eqsub} will be meant in the generalized sense,
i.e.~$W^{2,1}_{\textnormal{loc}}$-functions satisfying equation \eqref{eqsub} for a.e.~$t$.

Throughout the paper, we say that $u \in W^{2,1}_{\textnormal{loc}}(\mathbb{R})$ is a \textit{subharmonic solution of order $k$} of \eqref{eqsub} (with $k \geq 1$ an integer number)
if $u(t)$ is a $kT$-periodic solution of \eqref{eqsub} which is not $lT$-periodic for any integer $l = 1,\ldots,k-1$, that is, 
$kT$ is the minimal period of $u(t)$ in the class of the integer multiples of $T$. This is the most natural definition of subharmonic solutions
to equations like \eqref{eqsub}, when just the $T$-periodicity of $h(\cdot,u)$ is assumed; on the lines of \cite{MiTa-88}, if additional conditions on this time dependence are imposed, further information on the minimality of the period can be given (see Remark~\ref{rem-per}). 
Let us also notice that if $u(t)$ is a subharmonic solution of order $k$ of \eqref{eqsub}, the $k-1$ functions $u(\cdot + lT)$, for $l=1,\ldots,k-1$,
are subharmonic solutions of order $k$ of \eqref{eqsub} too; these solutions, though distinct,
have to be considered equivalent from the point of view of the counting of subharmonics. Accordingly, given $u_{1}(t), u_{2}(t)$ subharmonic solutions of order $k$ of \eqref{eqsub},
we say that $u_{1}(t)$ and $u_{2}(t)$ are not in the same \textit{periodicity class} if $u_{1}(\cdot) \not\equiv u_{2}(\cdot + lT)$ for any integer $l = 1,\ldots,k-1$.

Finally, we introduce the following notation. For any $q \in L^{1}(\mathopen{[}0,T\mathclose{]})$, we denote by $\lambda_{0}(q)$ the principal eigenvalue of the linear problem
\begin{equation}\label{eqhill}
v'' + (\lambda + q(t)) v = 0,
\end{equation}
with $T$-periodic boundary conditions. As well known (see, for instance, \cite[ch.~8, Theorem~2.1]{CoLe-55} and \cite[Theorem~2.1]{MaWi-66})
$\lambda_{0}(q)$ exists and is the unique real number such that the
linear equation \eqref{eqhill} admits one-signed $T$-periodic solutions. Recalling that, by definition, the \textit{Morse index} $m(q)$ of the linear equation 
$v'' + q(t)v = 0$ is the number of (strictly) negative $T$-periodic eigenvalues of \eqref{eqhill}, we immediately see that
$\lambda_{0}(q) < 0$ if and only if $m(q) \geq 1$.

\medskip

We are now in position to state the following result.

\begin{proposition}\label{propsub}
Let $h(t,u)$ be as above and assume that the global continuability for the solutions to \eqref{eqsub} is guaranteed. Moreover, suppose that:
\begin{itemize}
\item[$(i)$] there exists a $T$-periodic solution $u^{*}(t)$ of \eqref{eqsub} satisfying
\begin{equation}\label{hpmorse}
\lambda_{0}(\partial_u h(t,u^{*}(t))) < 0;
\end{equation}
\item[$(ii)$] there exists a $T$-periodic function $\alpha \in W^{2,1}_{\textnormal{loc}}(\mathbb{R})$ satisfying
\begin{equation}\label{low}
\alpha''(t) + h(t,\alpha(t)) \geq 0, \quad \text{for a.e. } t \in \mathbb{R},
\end{equation}
and
\begin{equation*}
\alpha(t) < u^{*}(t), \quad \text{for any } t \in \mathbb{R}.
\end{equation*}
\end{itemize}
Then there exists $k^{*} \geq 1$ such that for any integer $k \geq k^{*}$ there exists an integer $m_{k} \geq 1$ such that, for any integer $j$ relatively prime with $k$ and such that
$1 \leq j \leq m_{k}$, equation \eqref{eqsub} has two subharmonic solutions $u_{k,j}^{(1)}(t)$,  $u_{k,j}^{(2)}(t)$ of order $k$ (not belonging to the same periodicity class),
such that, for $i=1,2$, $u_{k,j}^{(i)}(t) - u^{*}(t)$ has exactly $2j$ zeros in the interval $\mathopen{[}0,kT\mathclose{[}$ and
\begin{equation}\label{localiz}
\alpha(t) \leq u_{k,j}^{(i)}(t), \quad \text{for any } t \in \mathbb{R}.
\end{equation}
\end{proposition}

Incidentally, we observe that Proposition~\ref{propsub} in particular ensures that equation \eqref{eqsub}
has two subharmonic solutions of order $k$ (not belonging to the same periodicity class) for any large integer $k$ (just, take $j=1$ in the above statement).

\begin{remark}\label{remalpha}
Let us recall that a $T$-periodic function $\alpha \in W^{2,1}_{\textnormal{loc}}(\mathbb{R})$ satisfying \eqref{low}
is a \textit{lower solution} for the $T$-periodic problem associated with \eqref{eqsub} (weaker notions of lower/upper solutions
could be introduced in the Carath\'{e}odory setting, see \cite{DCHa-96}). Clearly, if $\alpha(t)$ is a $T$-periodic solution of \eqref{eqsub}, then $\alpha(t)$ is a $T$-periodic lower solution;
in this case, due to the uniqueness for the Cauchy problems, \eqref{localiz} implies $\alpha(t) < u_{k,j}^{(i)}(t)$ for any $t$.
$\hfill\lhd$
\end{remark}

Proposition~\ref{propsub} is a variant of \cite[Theorem~2.2]{BoZa-13}\footnote{Actually, \cite[Theorem~2.2]{BoZa-13} deals with the symmetric case assuming the existence of
an upper solution $\beta(t) > u^{*}(t)$.}. However, some care is needed in comparing the two results.
First, \cite[Theorem~2.2]{BoZa-13} is stated for $h(t,u)$ smooth; the generalization to the Carath\'{e}odory setting
in this case is not completely straightforward. Second, the assumption corresponding to $(i)$ in \cite[Theorem~2.2]{BoZa-13} reads as
\begin{equation}\label{media}
\int_{0}^{T} \partial_u h(t,u^{*}(t))~\!dt > 0.
\end{equation}
The possibility of replacing this explicit condition with the abstract assumption $\lambda_{0}(\partial_u h(t,u^{*}(t)) < 0$ has been discussed
in \cite[Theorem~2.1]{BoOrZa-14}\footnote{Actually, in \cite[Theorem~2.1]{BoOrZa-14} the case $u^{*}(t) \equiv 0$ is taken into account;
however, the two situations are equivalent via a linear change of variable (see the proof of \cite[Proposition~1]{BoZa-13}).}.
Actually, in that paper the assumption $\rho(\partial_u h(t,u^{*}(t)) > 0$ is used, where 
$\rho(q)$ is the Moser rotation number (see \cite{Mo-81}) of the linear equation $v'' + q(t)v = 0$.
However, it is very well known in the theory of the Hill's equation
(see, for instance, \cite[Proposition~2.1]{GaZh-00}) that $\rho(q) > 0$ if and only if $\lambda_{0}(q) < 0$
(that is, if and only if the equation $v'' + q(t)v = 0$ is \textit{not disconjugate}). 

Related results, yielding the existence and multiplicity of harmonic (i.e.~$T$-periodic) solutions 
according to the interaction of the nonlinearity with (non-principal) eigenvalues, can be found in \cite{MaReTo-14,MaReZa-02,Za-03}.

\medskip

The complete proof of Proposition~\ref{propsub}, based on the Poincar\'{e}-Birkhoff fixed point theorem, is quite long.
For this reason, we provide just a sketch of it, referring to previous papers (in particular, to \cite{BoOrZa-14, BoZa-13}) for the most standard steps.

\begin{proof}[Sketch of the proof of Proposition~\ref{propsub}]
We define the truncated function
\begin{equation*}
\tilde h(t,u) :=
\begin{cases}
\, h(t,\alpha(t)), & \text{if } u \leq \alpha(t); \\
\, h(t,u), & \text{if } u > \alpha(t);
\end{cases}
\end{equation*}
and we set
\begin{equation*}
h^{*}(t,v) := \tilde h(t,u^{*}(t)+v) - h(t,u^{*}(t)), \quad \text{for any } (t,v) \in \mathbb{R}^{2}.
\end{equation*}
Then, we consider the equation
\begin{equation}\label{modif}
v'' + h^{*}(t,v) = 0.
\end{equation}
The following fact is easily proved, using maximum principle-type arguments (see \cite[p.~95]{BoZa-13}). 
\begin{itemize}
\item[$(\bigstar)$] If $v(t)$ is a \textit{sign-changing} $kT$-period solutions of \eqref{modif} (for some integer $k \geq 1$) then $v(t) \geq \alpha(t) - u^{*}(t)$ for any $t \in \mathbb{R}$.
\end{itemize}
Now, we observe that both uniqueness and global continuability for the solutions to the Cauchy problems associated with \eqref{modif} are ensured;
moreover, since $u^{*}(t) > \alpha(t)$, the constant function $v \equiv 0$ is a solution of \eqref{modif}.
We can therefore transform \eqref{modif} into an equivalent first order system in $\mathbb{R}^{2} \setminus \{0\}$,
passing to clockwise polar coordinates $v(t) = r(t) \cos\theta(t)$, $v'(t) = -r(t)\sin\theta(t)$.  

We claim that:
\begin{itemize}
\item[$(A1)$] there exists an integer $k^{*} \geq 1$ such that, for any integer $k \geq k^{*}$, there exist an integer $m_{k} \geq 1$ and $r_{*} > 0$ such that
any solution $(r(t),\theta(t))$ with $r(0) = r_{*}$ satisfies $\theta(kT)-\theta(0) > 2\pi m_{k}$;
\item[$(A2)$] for any integer $k \geq k^{*}$ there exists $R_{*} > r_{*}$ such that any solution $(r(t),\theta(t))$ with $r(0) = R_{*}$ satisfies $\theta(kT)-\theta(0) < 2\pi$.
\end{itemize}
From the above facts, it follows that the Poincar\'{e}-Birkhoff theorem (in the generalized version for non-invariant annuli, see \cite{Di-82,Re-97}) can be applied,
giving, for any $k \geq k^{*}$ and any $1 \leq j \leq m_{k}$, the existence of two $kT$-periodic solutions $v_{k,j}^{(i)}(t)$ ($i=1,2$)
to equation \eqref{modif} having exactly $2j$ zeros on $\mathopen{[}0,kT\mathclose{[}$. Using $(\bigstar)$, it is then immediate to see that $u_{k,j}^{(i)}(t) := v_{k,j}^{(i)}(t) + u^{*}(t)$
is a $kT$-periodic solution of \eqref{eqsub}, satisfying \eqref{localiz} and such that $u_{k,j}^{(i)}(t) - u^{*}(t)$ has exactly $2j$ zeros in the interval $\mathopen{[}0,kT\mathclose{[}$.
The fact that, for $j$ and $k$ relatively prime, $u_{k,j}^{(i)}(t)$ is a subharmonic solution of order $k$ is also easily verified, while $u_{k,j}^{(1)}(t)$ and
$u_{k,j}^{(2)}(t)$ are not in the same periodicity class due to a standard corollary of the Poincar\'{e}-Birkhoff theorem for the iterates of a map.
For more details on the application of this method, we refer to \cite{Bo-11,MaReZa-02,Za-03}.

To conclude the proof, we then have to verify the claims $(A1)$ and $(A2)$. As for the first one, it can be proved exactly as in \cite[Proof of Theorem~2.1]{BoOrZa-14}
(see also \cite[Remark~2.2]{BoOrZa-14}).
The fact that we are working in a Carath\'{e}odory setting does not cause here serious difficulties, since the dominated convergence theorem
easily yields $h^{*}(\cdot,v)/v \to \partial_{u} h(\cdot,u^{*}(\cdot))$ in $L^{1}(\mathopen{[}0,T\mathclose{]})$ for $v \to 0$, and this is enough to use 
continuous dependence arguments as in \cite{BoOrZa-14}.
On the other hand, the proof of $(A2)$ is more delicate (especially when dealing with Carath\'{e}odory functions)
and we prefer to give some more details. We are going to use a trick based on modified polar coordinates, introduced in \cite{Fa-87} (see also \cite{Bo-11}).
More precisely, for any $\mu > 0$, we write
\begin{equation*}
v(t) = \dfrac{r_{\mu}(t)}{\mu} \cos\theta_{\mu}(t), \qquad v'(t) = -r_{\mu}(t)\sin\theta_{\mu}(t);
\end{equation*}
for further convenience we also compute
\begin{equation}\label{thetaprimo}
\theta_{\mu}'(t) = \mu \dfrac{v'(t)^{2} - v(t)v''(t)}{\mu^{2} v(t)^{2} + v'(t)^{2}}.
\end{equation}
The angular coordinates $\theta$ and $\theta_{\mu}$ are in general different.
However, the angular width of any quadrant of the plane is $\pi/2$ also if measured using the angle $\theta_{\mu}$.
As a consequence, recalling \eqref{thetaprimo} we can write the formula
\begin{equation}\label{stima}
\dfrac{1}{4} = \dfrac{\mu}{2\pi}\int_{t_{1}}^{t_{2}}\dfrac{v'(t)^{2} - v(t)v''(t)}{\mu^{2} v(t)^{2} + v'(t)^{2}} ~\!dt,
\end{equation}
valid whenever $t_{1}, t_{2}$ are such that $t_{1} < t_{2}$, $v(t_{1}) = 0 = v'(t_{2})$ (or viceversa) and $(v(t),v'(t))$ belongs to the same quadrant for $t \in \mathopen{[}t_{1},t_{2}\mathclose{]}$.
We stress that \eqref{stima} holds for any $\mu > 0$.

We can now give the proof. Preliminarily, we observe that, using the Cara\-th\'{e}odory condition together with the definition of $h^{*}(t,v)$, 
we can obtain
\begin{equation}\label{stimacar}
\vert h^{*}(t,v) \vert \leq b(t), \quad \text{for a.e. } t \in \mathopen{[}0,T\mathclose{]}, \; \text{for any } v \leq 0,
\end{equation}
where $b \in L^{1}(\mathopen{[}0,T\mathclose{]})$. We now fix an integer $k \geq k^{*}$ and take $\mu > 0$ so small that
\begin{equation}\label{t1}
\dfrac{\mu kT}{2\pi} \leq \dfrac{1}{16}.
\end{equation}
In view of the global continuability of the solutions, there exists $R_{*} > 0$ large enough such that
$r(0) = R_{*}$ implies that 
\begin{equation}\label{t2}
r_{\mu}(t)^{2} = \mu^{2} v(t)^{2} + v'(t)^{2} \geq \biggl{(}\dfrac{8 k \Vert b \Vert_{L^{1}(\mathopen{[}0,T\mathclose{]})}}{\pi}\biggr{)}^{2}, \quad \text{for any } t \in \mathopen{[}0,kT\mathclose{]}.
\end{equation}
At this point, assume by contradiction that $\theta(kT) - \theta(0) \geq 2\pi$ for a solution with 
$r(0) = R_{*}$. Then it is not difficult to see that there exist $t_{1}, t_{2} \in \mathopen{[}0,kT\mathclose{]}$ with $t_{1} < t_{2}$ and such that either
$v(t_{1}) = 0 = v'(t_{2})$ and $(v(t),v'(t))$ belongs to the third quadrant for $t \in \mathopen{[}t_{1},t_{2}\mathclose{]}$ or 
$v'(t_{1}) = 0 = v(t_{2})$ and $(v(t),v'(t))$ belongs to the fourth quadrant for $t \in \mathopen{[}t_{1},t_{2}\mathclose{]}$.
As a consequence, on one hand \eqref{stima} holds true; on the other hand, since $v(t) \leq 0$ for $t \in \mathopen{[}t_{1},t_{2}\mathclose{]}$ we can use 
\eqref{stimacar} so as to obtain
\begin{equation*}
\vert v(t)v''(t) \vert \leq b(t)\vert v(t) \vert, \quad \text{for a.e. } t \in \mathopen{[}t_{1},t_{2}\mathclose{]}.
\end{equation*}
Combining these two facts, we find
\begin{align*}
\dfrac{1}{4} & \leq \dfrac{\mu}{2\pi}\int_{t_{1}}^{t_{2}} \dfrac{v'(t)^{2}}{\mu^{2} v(t)^{2} + v'(t)^{2}}~\!dt 
+ \dfrac{1}{2\pi}\int_{t_{1}}^{t_{2}} \dfrac{b(t) r_{\mu}(t) \vert \cos\theta_{\mu}(t)\vert}{r_{\mu}(t)^{2}}~\!dt \\
& \leq \dfrac{\mu kT}{2\pi} + \dfrac{k \Vert b \Vert_{L^{1}(\mathopen{[}0,T\mathclose{]})}}{2\pi} \dfrac{1}{\min_{t \in \mathopen{[}0,kT\mathclose{]}} r_{\mu}(t)}.
\end{align*}
Using \eqref{t1} and \eqref{t2}, we finally find $\tfrac{1}{4} \leq \tfrac{1}{16} + \tfrac{1}{16} = \tfrac{1}{8}$, a contradiction.
\end{proof}

\begin{remark}\label{confronto}
It is worth noticing that, although related, conditions \eqref{hpmorse} and \eqref{media} are not equivalent. More precisely, given a general weight $q \in L^{1}(\mathopen{[}0,T\mathclose{]})$, 
\begin{equation}\label{conf}
\int_{0}^{T} q(t)~\!dt > 0 \quad \Longrightarrow \quad \lambda_{0}(q) < 0
\end{equation}
as an easy consequence of the variational characterization of the principal eigenvalue (see \cite[Theorem~4.2]{MaWi-66})
\begin{equation}\label{varlambda}
\lambda_{0}(q) = \inf_{v \in H^{1}_{T}}\dfrac{\int_{0}^{T} \bigl{(} v'(t)^{2} - q(t)v(t)^{2}\bigr{)}~\!dt}{\int_{0}^{T} v(t)^{2}~\!dt}
\end{equation}
(just, take $v \equiv 1$ in the above formula; $H^{1}_{T}$ denotes the Sobolev space of $T$-periodic $H^{1}_{\textnormal{loc}}$-functions).
Of course, \eqref{varlambda} also implies that
\begin{equation*}
q(t) \leq 0 \quad \Longrightarrow \quad \lambda_{0}(q) \geq 0
\end{equation*}
but there exist (sign-changing) weights $q(t)$ such that $\int_{0}^{T} q(t)~\!dt \leq 0$ and 
$\lambda_{0}(q) < 0$, showing that the converse of \eqref{conf} is not true. Explicit examples can be constructed, for instance, as in \cite[Remark~3.5]{Bo-11}.
An even more interesting example will be given later (see Remark~\ref{rem-indmorse1}), showing that the possibility of replacing \eqref{media} 
with the weaker assumption \eqref{hpmorse} is crucial for our purposes.
$\hfill\lhd$
\end{remark}

\section{Statement of the main results}\label{section-3}

In this section, we state our main results, dealing with positive solutions to equations of the type
\begin{equation}\label{eqmain}
u'' + a(t)g(u) = 0.
\end{equation}
We always assume that $g \in \mathcal{C}^{2}(I)$, with $I\subseteq\mathbb{R}^{+}:=\mathopen{[}0,+\infty\mathclose{[}$ a right neighborhood of $s = 0$, and satisfies the following conditions:
\begin{equation*}
g(0) = 0
\leqno{(g_{1})}
\end{equation*}
\begin{equation*}
g'(0) = 0
\leqno{(g_{2})}
\end{equation*}
\begin{equation*}
g''(s) > 0, \quad \text{for every } s\in I\setminus\{0\}.
\leqno{(g_{3})}
\end{equation*}
Hence, $g(s)$ is \textit{superlinear at zero} and \textit{strictly convex}. Incidentally, notice that from $(g_{1})$, $(g_{2})$ and $(g_{3})$
it follows that $g(s)$ is strictly increasing; in particular
\begin{equation*}
g(s) > 0, \quad \text{for every } s\in I\setminus\{0\},
\end{equation*}
implying that the only constant solution to \eqref{eqmain} is the trivial one, i.e.~$u \equiv 0$.

As for the weight function, we suppose that $a \colon \mathbb{R} \to \mathbb{R}$ is a $T$-periodic and locally integrable function satisfying the following condition:
\begin{itemize}
\item [$(a_{1})$]
\textit{there exist $m\geq 1$ intervals $I^{+}_{1},\ldots,I^{+}_{m}$, closed and pairwise disjoint in the quotient space $\mathbb{R}/T\mathbb{Z}$, such that
\begin{align*}
& a(t)\geq0, \; \text{ for a.e. } t\in I^{+}_{i}, \quad a(t)\not\equiv0 \; \text{ on } I^{+}_{i}, \quad \text{for } i=1,\ldots,m; \\
& a(t)\leq0, \; \text{ for a.e. } t\in(\mathbb{R}/T\mathbb{Z})\setminus\bigcup_{i=1}^{m}I^{+}_{i}.
\end{align*}
}
\end{itemize}
Moreover, motivated by the discussion in the introduction, we suppose that the mean value condition
\begin{equation*}
\int_{0}^{T} a(t)~\!dt < 0
\leqno{(a_{2})}
\end{equation*}
holds true.

Of course, by a solution to equation \eqref{eqmain} we mean a function $u \in W^{2,1}_{\textnormal{loc}}$, with $u(t) \in I$ for any $t$ and solving 
\eqref{eqmain} for a.e.~$t$. Notice that, since $I \subseteq \mathbb{R}^+$, any solution is a non-negative function; we say that a solution is positive if $u(t) > 0$ for any $t$.  

\medskip

As a first result, we provide a statement generalizing the one given in the introduction for equation \eqref{eqmod}.
More precisely, we show that the existence of positive subharmonic solutions (in the sense clarified at the beginning of Section~\ref{section-2}, see also Remark~\ref{rem-per})
to \eqref{eqmain} is ensured for any function $g(s)$
which satisfies $(g_{1})$, $(g_{2})$, $(g_{3})$ for $I = \mathbb{R}^{+}$ and which is \textit{superlinear at infinity}. Needless to say, this is the case for the model nonlinearity $g(s) = s^{p}$ with $p > 1$.

\begin{theorem}\label{th1}
Let $a \colon \mathbb{R} \to \mathbb{R}$ be a $T$-periodic locally integrable function satisfying $(a_{1})$ and $(a_{2})$.
Let $g \in \mathcal{C}^{2}(\mathbb{R}^{+})$ satisfy $(g_{1})$, $(g_{2})$ and $(g_{3})$, as well as 
\begin{equation*}
\lim_{s \to +\infty}\dfrac{g(s)}{s} = +\infty.
\leqno{(g_{4})}
\end{equation*}
Then, there exists a positive $T$-periodic solution $u^{*}(t)$ of \eqref{eqmain};
moreover, there exists $k^{*} \geq 1$ such that for any integer $k \geq k^{*}$ there exists an integer $m_{k} \geq 1$ such that, for any integer $j$ relatively prime with $k$ and such that
$1 \leq j \leq m_{k}$, equation \eqref{eqmain} has two positive subharmonic solutions $u_{k,j}^{(i)}(t)$ ($i=1,2$) of order $k$ (not belonging to the same periodicity class),
such that $u_{k,j}^{(i)}(t) - u^{*}(t)$ has exactly $2j$ zeros in the interval $\mathopen{[}0,kT\mathclose{[}$.
\end{theorem}

In our second result, we deal with the case $I = \mathopen{[}0,\delta\mathclose{[}$, with $\delta > 0$ finite, assuming a \textit{singular} behavior for $g(s)$ when $s \to \delta^{-}$.

\begin{theorem}\label{th2}
Let $a \colon \mathbb{R} \to \mathbb{R}$ be a $T$-periodic locally integrable function satisfying $(a_{1})$ and $(a_{2})$.
Let $g \in \mathcal{C}^{2}(\mathopen{[}0,\delta\mathclose{[})$ (for some $\delta > 0$ finite) satisfy 
$(g_{1})$, $(g_{2})$ and $(g_{3})$, as well as
\begin{equation*}
\lim_{s \to \delta^-}g(s) = +\infty.
\leqno{(g_{4}')}
\end{equation*}
Then, there exists a positive $T$-periodic solution $u^{*}(t)$ of \eqref{eqmain};
moreover, there exists $k^{*} \geq 1$ such that for any integer $k \geq k^{*}$ there exists an integer $m_{k} \geq 1$ such that, for any integer $j$ relatively prime with $k$ and such that
$1 \leq j \leq m_{k}$, equation \eqref{eqmain} has two positive subharmonic solutions $u_{k,j}^{(i)}(t)$ ($i=1,2$) of order $k$ (not belonging to the same periodicity class),
such that $u_{k,j}^{(i)}(t) - u^{*}(t)$ has exactly $2j$ zeros in the interval $\mathopen{[}0,kT\mathclose{[}$.
\end{theorem}

We mention that singular equations with indefinite weight were considered in \cite{BrTo-10,Ur-pp2015,Ur-16}. More precisely, these papers deal with equations like
$u'' + a(t)/u^{\sigma} = 0$, where $\sigma > 0$. Our setting is different and Theorem~\ref{th2} applies for instance to the equation 
\begin{equation}\label{eqsing}
u'' + a(t)\biggl{(}\dfrac{u^{\gamma}}{1-u^{\sigma}}\biggr{)} = 0,
\end{equation}
for $\gamma > 1$ and $\sigma \geq 1$. To the best of our knowledge, even the mere existence of a positive $T$-periodic solution to \eqref{eqsing}
is a fact which has never been noticed.

Finally, we give a purely \textit{local result}. More precisely, we just assume $(g_{1})$, $(g_{2})$ and $(g_{3})$ in a bounded interval $I = \mathopen{[}0,\rho\mathclose{]}$, with $\rho > 0$ finite;
on the other hand, we deal with an equation depending on a real parameter and we manage to obtain the result by varying it.

\begin{theorem}\label{th3}
Let $a \colon \mathbb{R} \to \mathbb{R}$ be a $T$-periodic locally integrable function satisfying $(a_{1})$ and $(a_{2})$.
Let $g \in \mathcal{C}^{2}(\mathopen{[}0,\rho\mathclose{]})$ (for some $\rho > 0$) satisfy 
$(g_{1})$, $(g_{2})$ and $(g_{3})$.
Then, there exists $\lambda^{*} > 0$ such that for any $\lambda > \lambda^{*}$ 
there exists a positive $T$-periodic solution $u^{*}(t)$ of the parameter-dependent equation
\begin{equation}\label{eqpar}
u'' + \lambda a(t)g(u) = 0
\end{equation}
satisfying $\max_{t \in \mathbb{R}}u^{*}(t) < \rho$. Moreover, 
there exists $k^{*} \geq 1$ such that for any integer $k \geq k^{*}$ there exists an integer $m_{k} \geq 1$ such that, for any integer $j$ relatively prime with $k$ and such that
$1 \leq j \leq m_{k}$, equation \eqref{eqpar} has two positive subharmonic solutions $u_{k,j}^{(i)}(t)$ ($i=1,2$) of order $k$ (not belonging to the same periodicity class),
with $\max_{t \in \mathbb{R}}u_{k,j}^{(i)}(t) < \rho$ and such that $u_{k,j}^{(i)}(t) - u^{*}(t)$ has exactly $2j$ zeros in the interval $\mathopen{[}0,kT\mathclose{[}$.
\end{theorem}

Of course, in the above statement $g(s)$ may be defined also for $s > \rho$, but no assumptions on its behavior are made.
For instance, we can apply Theorem~\ref{th3} to parameter-dependent equations like
\begin{equation}\label{eq-par}
u'' + \lambda a(t) \biggl{(} \dfrac{u^{\gamma}}{1 + u^{\sigma}} \biggr{)} = 0,
\end{equation}
with $\sigma \geq \gamma -1 > 0$, obtaining the following: for any $\rho > 0$ small enough, there exists $\lambda^{*} = \lambda^{*}(\rho) > 0$ such that for any $\lambda > \lambda^{*}$ equation
\eqref{eq-par} has a positive $T$-periodic solution as well as positive subharmonic solutions of any large order; all these periodic solutions, moreover, have maximum less than $\rho$. 
In such a way, we can complement - in the direction of proving the existence of positive subharmonics - 
recent results dealing with positive harmonic solutions in
the asymptotically linear case $\sigma = \gamma -1$
(see \cite[Corollary~3.7]{FeZa-15ade}) and in the sublinear one $\sigma > \gamma-1$ (see \cite{BoFeZa-16,BoZa-12}).
It is worth noticing that, according to \cite[Theorem~4.3]{BoFeZa-16}, in this latter case a further positive $T$-periodic solution (having maximum greater than $\rho$) to \eqref{eq-par} appears.
This second solution is expected to have typically zero Morse index, and no positive subharmonic solutions oscillating around it.

\begin{remark}\label{rem-per}
We notice that all the positive subharmonic solutions of order $k$ found in this section actually have minimal period $kT$ if we further assume that $T > 0$ is the \emph{mimimal} period of $a(t)$. This is easily seen, by writing \eqref{eqmain} in the equivalent form $a(t) = - u''(t)/g(u(t))$.   
$\hfill\lhd$
\end{remark}

\section{Proof of the main results}\label{section-4}

In this section we provide the proof of the results presented in Section~\ref{section-3}.
Actually, we are going to give and prove a further statement,
which looks slightly more technical but has the advantage of unifying all the situations considered in Theorem~\ref{th1}, Theorem~\ref{th2} and Theorem~\ref{th3}.

Henceforth, we deal with the equation
\begin{equation}\label{eq-proof}
u'' + a(t)f(u) = 0,
\end{equation}  
where $f \in \mathcal{C}^{2}(\mathopen{[}0,\rho\mathclose{]})$, for some $\rho > 0$ finite, and satisfies:
\begin{equation*}
f(0) = 0
\leqno{(f_{1})}
\end{equation*}
\begin{equation*}
f'(0) = 0
\leqno{(f_{2})}
\end{equation*}
\begin{equation*}
f''(s) > 0, \quad \text{for every } s \in \mathopen{]}0,\rho\mathclose{]}.
\leqno{(f_{3})}
\end{equation*}
Accordingly, by a solution to equation \eqref{eq-proof} we mean a function $u \in W^{2,1}_{\textnormal{loc}}$, with $0\leq u(t) \leq \rho$ for any $t$ and solving 
\eqref{eq-proof} in the Carath\'{e}odory sense; a solution is said to be positive if $u(t) > 0$ for any $t$.  

\medskip

In this setting, the following result can be given.

\begin{theorem}\label{th-proof}
Let $a \colon \mathbb{R} \to \mathbb{R}$ be a $T$-periodic locally integrable function satisfying $(a_{1})$ and $(a_{2})$.
Then there exist two real constants $M_{1}\in\mathopen{]}0,1\mathclose{[}$ and $M_{2}>0$ such that,
for any $\rho > 0$ and for any $f \in \mathcal{C}^{2}(\mathopen{[}0,\rho\mathclose{]})$ satisfying $(f_{1})$, $(f_{2})$, $(f_{3})$ and 
\begin{equation*}
\dfrac{f(M_{1}\rho)}{M_{1}\rho} > M_{2},
\leqno{(f_{4})}
\end{equation*}
the following holds true: there exists a positive $T$-periodic solution $u^{*}(t)$ of \eqref{eq-proof} with $\max_{t \in \mathbb{R}}u^{*}(t) < \rho$;
moreover, there exists $k^{*} \geq 1$ such that for any integer $k \geq k^{*}$ there exists an integer $m_{k} \geq 1$ such that, for any integer $j$ relatively prime with $k$ and such that
$1 \leq j \leq m_{k}$, equation \eqref{eq-proof} has two positive subharmonic solutions $u_{k,j}^{(i)}(t)$ ($i=1,2$) of order $k$ (not belonging to the same periodicity class),
with $\max_{t \in \mathbb{R}}u_{k,j}^{(i)}(t) < \rho$ and such that $u_{k,j}^{(i)}(t) - u^{*}(t)$
has exactly $2j$ zeros in the interval $\mathopen{[}0,kT\mathclose{[}$.
\end{theorem}

It is clear that all the theorems in Section~\ref{section-3} follows from Theorem~\ref{th-proof}. 
More precisely:
\begin{itemize}
\item
in order to obtain Theorem~\ref{th1}, we take $f = g$ and $\rho > 0$ large enough: then $(f_{1})$, $(f_{2})$, $(f_{3})$ correspond to
$(g_{1})$, $(g_{2})$, $(g_{3})$, while $(f_{4})$ comes from $(g_{4})$;
\item
in order to obtain Theorem~\ref{th2}, we take $f = g$ and $\rho < \delta$ with $\delta - \rho$ small enough: then $(f_{1})$, $(f_{2})$, $(f_{3})$ correspond to
$(g_{1})$, $(g_{2})$, $(g_{3})$, while $(f_{4})$ comes from $(g_{4}')$;
\item
in order to obtain Theorem~\ref{th3}, we take $f = \lambda g$: then $(f_{1})$, $(f_{2})$, $(f_{3})$ correspond to
$(g_{1})$, $(g_{2})$, $(g_{3})$ (independently on $\lambda > 0$), while $(f_{4})$ is certainly satisfied for $\lambda > 0$ large enough.
\end{itemize}

\medskip

Now we are going to prove Theorem~\ref{th-proof}. Wishing to apply Proposition~\ref{propsub}, we proceed as follows.
First, we define an extension $\hat{f}(s)$ of $f(s)$ for $s \geq \rho$, having linear growth at infinity and thus ensuring the global continuability of the
(positive) solutions of $u'' + a(t)\hat{f}(u)$; in doing this, we need to check that any periodic solution of this modified equation is actually smaller
than $\rho$, thus solving the original equation $u'' + a(t)f(u) = 0$. This is the most technical part of the proof (producing the constants $M_{1}, M_{2}$
appearing in assumption $(f_{4})$) and is developed in Section~\ref{section-4.1}.
Second, in Section~\ref{section-4.2}, using a degree theoretic approach (and taking advantage of the a-priori bound given in the previous section),
we prove the existence of a positive $T$-periodic solution of $u'' + a(t)f(u) = 0$.
Third, in Section~\ref{section-4.3} we provide the desired Morse index information.
The easy conclusion of the proof is finally given in Section~\ref{section-4.4} (we just notice here that the existence of a lower solution $\alpha(t) < u^{*}(t)$
is straightforward, since we can take $\alpha(t) \equiv 0$).

It is worth noticing that condition $(f_{3})$ (requiring in particular that $f \in \mathcal{C}^{2}$) will be essential only in Section~\ref{section-4.3}. 
For this reason, we carry out the discussion in Section~\ref{section-4.1} and Section~\ref{section-4.2} (containing results which may have some independent interests)
in a slightly more general setting than the one in Theorem~\ref{th-proof}.

\subsection{The a-priori bound}\label{section-4.1}

In this section, we prove an a-priori bound valid for periodic solutions of \eqref{eq-proof}
as well as for periodic solutions of a related equation (see \eqref{eq-lem-4.2} below).
This will be useful both for the application of Proposition~\ref{propsub} (requiring a globally defined nonlinearity) and for the degree approach discussed in the Section~\ref{section-4.2}. 
 
As already anticipated, in this section we do not assume all the conditions on $f(s)$ required in Theorem~\ref{th-proof}. More precisely,
we are going to deal with continuously differentiable functions $f \colon \mathopen{[}0,\rho\mathclose{]} \to \mathbb{R}^{+}$ satisfying $(f_{1})$ and the following condition
\begin{equation*}
f(s) > 0, \quad  \text{for every } s \in \mathopen{]}0,\rho\mathclose{]}.
\leqno{(f_{*})}
\end{equation*}
Moreover, instead of $(f_{3})$ we just suppose that $f(s)$ is a \textit{convex function}, namely
\begin{equation*}
f(\vartheta s_{1}+(1-\vartheta)s_{2}) \leq \vartheta f(s_{1}) + (1-\vartheta) f(s_{2}),
\quad \forall \, s_{1},s_{2}\in\mathopen{[}0,\rho\mathclose{]}, \; \forall \, \vartheta\in\mathopen{[}0,1\mathclose{]}.
\end{equation*}

For further convenience, we observe that from the above conditions it follows that $f(s)$ is 
non-decreasing and such that $s\mapsto f(s)/s$ is a non-decreasing map in $\mathopen{]}0,\rho\mathclose{]}$.
Indeed, let $0<s_{1}<s_{2}$ and let $\vartheta\in\mathopen{[}0,1\mathclose{]}$ be such that $s_{1}=\vartheta s_{2}$. Then, we have
\begin{equation*}
f(s_{1}) = f(\vartheta s_{2} + (1-\vartheta)0) \leq \vartheta f(s_{2}) + (1-\vartheta) f(0) = \dfrac{s_{1}}{s_{2}}f(s_{2})
\end{equation*}
and thus the map $s\mapsto f(s)/s$ is non-decreasing in $\mathopen{]}0,\rho\mathclose{]}$.
Consequently, we immediately obtain that $s\mapsto f(s)$ is a non-decreasing map in $\mathopen{[}0,\rho\mathclose{]}$,
since it is the product of two non-decreasing positive maps in $\mathopen{]}0,\rho\mathclose{]}$. 

We also recall that a function $f\in\mathcal{C}^{1}(\mathopen{[}0,\rho\mathclose{]})$ is convex if and only if
$f(s)$ lies above all of its tangents, hence
\begin{equation*}
f(s_{1})\geq f(s_{2})+f'(s_{2})(s_{1}-s_{2}), \quad \forall \, s_{1},s_{2}\in\mathopen{[}0,\rho\mathclose{]}.
\end{equation*}
Using $(f_{1})$, $(f_{*})$ and the above inequality (with $s_{1}=0$ and $s_{2}=\rho$), we immediately obtain
\begin{equation*}
f'(\rho) \geq \dfrac{f(\rho)}{\rho} > 0.
\end{equation*}

With this in mind, we introduce the extension $\hat{f} \colon \mathbb{R}^{+} \to \mathbb{R}$ defined as
\begin{equation}\label{def-hat}
\hat{f}(s):=
\begin{cases}
\, f(s), & \text{if } s\in \mathopen{[}0,\rho\mathclose{]}; \\
\, f(\rho)+f'(\rho)(s-\rho), & \text{if } s\in \mathopen{]}\rho,+\infty\mathclose{[}.
\end{cases}
\end{equation}
It is easily seen that the map $\hat{f}(s)$ is continuously differentiable, convex, non-decreasing and such that $\hat{f}(s)>0$ for all $s>0$.
Then, arguing as above, we immediately obtain that the map $s\mapsto\hat{f}(s)/s$ is non-decreasing as well.

\medskip

We are now in a position to state our technical result (whose proof benefits from some arguments developed in \cite[p.~421]{GaHaZa-03}
and in \cite[Lemma~4.1]{BoFeZa-pp2015}) giving a-priori bounds for periodic solutions of the equation 
\begin{equation}\label{eq-lem-4.2}
u'' + a(t)\hat{f}(u) + \nu \mathbbm{1}_{\bigcup_{i=1}^{m} I^{+}_{i}}(t) = 0,
\end{equation}
where $\nu \geq 0$ and $\mathbbm{1}_{\bigcup_{i=1}^{m} I^{+}_{i}}(t)$ denotes the indicator function of the set $\bigcup_{i=1}^{m} I^{+}_{i}$. 
Incidentally, notice that neither the mean value condition $(a_{2})$ nor the superlinearity assumption at zero $(f_{2})$ are required in the statement below.

\begin{lemma}\label{lem-4.1}
Let $a \colon \mathbb{R} \to \mathbb{R}$ be a $T$-periodic locally integrable function satisfying $(a_{1})$.
Then there exist two real constants $M_{1}\in\mathopen{]}0,1\mathclose{[}$ and $M_{2}>0$ such that,
for every $\rho>0$, for every convex function $f\in\mathcal{C}^{1}(\mathopen{[}0,\rho\mathclose{]})$ satisfying
$(f_{1})$, $(f_{*})$ and $(f_{4})$, for every $\nu\geq0$ and for every integer $k\geq1$, any $kT$-periodic solution $u(t)$ to \eqref{eq-lem-4.2}
satisfies $\max_{t \in \mathbb{R}} u(t) < \rho$.
\end{lemma}

\begin{proof}
According to condition $(a_{1})$, we can find $2m + 1$ points
\begin{equation*}
\sigma_{1} < \tau_{1} < \ldots < \sigma_{i} < \tau_{i} < \ldots < \sigma_{m} < \tau_{m} < \sigma_{m+1}, \quad \text{with } \; \sigma_{m+1}-\sigma_{1} = T,
\end{equation*}
such that
\begin{equation*}
I^{+}_{i} = \mathopen{[}\sigma_{i},\tau_{i}\mathclose{]}, \qquad i=1,\ldots,m.
\end{equation*}
We then fix $\varepsilon > 0$ such that
\begin{equation*}
\varepsilon < \dfrac{|I^{+}_{i}|}{2}
\quad \text{ and } \quad
\int_{\sigma_{i}+\varepsilon}^{\tau_{i}-\varepsilon} a^{+}(t)~\!dt > 0,
\qquad \text{for all } i\in\{1,\ldots,m\};
\end{equation*}
so that the constant
\begin{equation*}
\eta_{\varepsilon} := \min_{i=1,\ldots,m} \int_{\sigma_{i}+\varepsilon}^{\tau_{i}-\varepsilon} a^{+}(t)~\!dt
\end{equation*}
is well defined and positive.
Next, we define the constants
\begin{equation*}
M_{1}:= \dfrac{\varepsilon}{\displaystyle \max_{i=1,\ldots,m}|I^{+}_{i}|}
\quad \text{ and } \quad
M_{2}:= \dfrac{2}{M_{1} \varepsilon \eta_{\varepsilon}}.
\end{equation*}
Notice that $M_{1}\in\mathopen{]}0,1\mathclose{[}$, since $\varepsilon < |I^{+}_{i}|$, for all $i=1,\ldots,m$. We stress that $M_{1}$ and $M_{2}$ depend only on the weight function $a(t)$.

Let us consider an arbitrary $\nu \geq 0$ and an arbitrary convex function $f\in\mathcal{C}^{1}(\mathopen{[}0,\rho\mathclose{]})$ satisfying $(f_{1})$, $(f_{*})$ and $(f_{4})$.
By contradiction, we suppose that $u(t)$ is a $kT$-periodic solution of \eqref{eq-lem-4.2} such that
\begin{equation*}
\max_{t\in \mathbb{R}} u(t) =: \rho^{*} \geq \rho.
\end{equation*}
Setting $I^{+}_{i,\ell}:=I^{+}_{i}+\ell T$ (for $i= 1,\ldots,m$ and $\ell\in\mathbb{Z}$),
the convexity of $u(t)$ on $\mathbb{R}\setminus\bigcup_{i,\ell}I^{+}_{i,\ell}$
ensures that the maximum is attained in some $I^{+}_{i,\ell}$.
Accordingly, we can suppose that there is an index $i\in\{1,\ldots,m\}$ and $\ell\in\{0,\ldots,k-1\}$ such that
\begin{equation*}
\max_{t\in I^{+}_{i,\ell}} u(t) = \rho^{*}.
\end{equation*}
Up to a relabeling of the intervals $I^{+}_{i,\ell}$, we can also suppose $\ell=0$ (notice that the constants $M_{1}$ and $M_{2}$ do not change since $a(t)$ is $T$-periodic).
From now on, we therefore assume that 
\begin{equation*}
\max_{t\in I^{+}_{i}} u(t) = \rho^{*}.
\end{equation*}
From this fact, together with the concavity of $u(t)$ on $I^{+}_{i}$, we obtain
\begin{equation*}
u(t)\geq \dfrac{\rho^{*}}{|I^{+}_{i}|}\min\{t-\sigma_{i},\tau_{i}-t\}, \quad \forall \, t\in I^{+}_{i},
\end{equation*}
(cf.~\cite[p.~420]{GaHaZa-03} for a similar estimate)
and hence
\begin{equation}\label{eq-4.4}
u(t)\geq \dfrac{\varepsilon\rho^{*}}{\displaystyle \max_{i=1,\ldots,m}|I^{+}_{i}|} = M_{1}\rho^{*}, \quad \forall \, t\in I^{+}_{i}.
\end{equation}
On the other hand, we claim that
\begin{equation*}
|u'(t)| \leq \dfrac{u(t)}{\varepsilon}, \quad \forall \, t\in \mathopen{[}\sigma_{i}+\varepsilon,\tau_{i}-\varepsilon\mathclose{]}.
\end{equation*}
Indeed, if $t\in \mathopen{[}\sigma_{i}+\varepsilon,\tau_{i}-\varepsilon\mathclose{]}$ is such that $u'(t)=0$, the result is trivial.
If $u'(t)>0$, again using the concavity of $u(t)$ on $I^{+}_{i}$, we have
\begin{equation*}
u(t) \geq u(t) - u(\sigma_{i}) = \int_{\sigma_{i}}^{t} u'(\xi)~\!d\xi
\geq u'(t) (t-\sigma_{i}) \geq u'(t)\varepsilon.
\end{equation*}
Analogously, if $u'(t)<0$, we have
\begin{equation*}
u(t) \geq u(t) - u(\tau_{i}) = - \int_{t}^{\tau_{i}} u'(\xi)~\!d\xi
\geq -u'(t) (\tau_{i}-t) \geq - u'(t)\varepsilon.
\end{equation*}
The claim is thus proved. As a consequence,
\begin{equation}\label{eq-4.5}
|u'(t)| \leq \dfrac{\rho^{*}}{\varepsilon}, \quad \forall \, t\in \mathopen{[}\sigma_{i}+\varepsilon,\tau_{i}-\varepsilon\mathclose{]}.
\end{equation}

Integrating equation \eqref{eq-lem-4.2} on $\mathopen{[}\sigma_{i}+\varepsilon,\tau_{i}-\varepsilon\mathclose{]}$
and using \eqref{eq-4.4}, \eqref{eq-4.5} and the monotonicity of $s\mapsto \hat{f}(s)$, we have
\begin{align*}
\hat{f}(M_{1}\rho^{*}) \int_{\sigma_{i}+\varepsilon}^{\tau_{i}-\varepsilon} a^{+}(t)~\!dt
   &\leq \int_{\sigma_{i}+\varepsilon}^{\tau_{i}-\varepsilon} a^{+}(t) \hat{f}(u(t))~\!dt =
   \int_{\sigma_{i}+\varepsilon}^{\tau_{i}-\varepsilon}\bigl{(} - u''(t) - \nu \bigr{)}~\!dt
\\ &= u'(\sigma_{i}+\varepsilon) - u'(\tau_{i}-\varepsilon) - \nu \, (\tau_{i}-\sigma_{i}-2\varepsilon)
\leq \dfrac{2\rho^{*}}{\varepsilon}.
\end{align*}
Dividing by $M_{1}\rho^{*}\eta_{\varepsilon}$ the above inequality and using the monotonicity of the map $s\mapsto \hat{f}(s)/s$, we obtain 
\begin{equation*}
\dfrac{f(M_{1}\rho)}{M_{1}\rho} = \dfrac{\hat{f}(M_{1}\rho)}{M_{1}\rho} \leq \dfrac{\hat{f}(M_{1}\rho^{*})}{M_{1}\rho^{*}} \leq \dfrac{2}{M_{1} \varepsilon \eta_{\varepsilon}} = M_{2},
\end{equation*}
a contradiction with respect to hypothesis $(f_{4})$.
\end{proof}

\begin{remark}\label{rem-4.a}
It is worth noticing that, in the above proof, the fact that $u(t)$ is $kT$-periodic is used only to ensure, via a convexity argument,
that its maximum is achieved in some positivity interval $I^{+}_{i}$. Accordingly, it is easily seen that the conclusion of Lemma~\ref{lem-4.1} still holds true for any globally defined bounded solution of \eqref{eq-lem-4.2}, as well as for solutions defined on compact intervals and satisfying Dirichlet/Neumann conditions at the boundary. 
$\hfill\lhd$
\end{remark}

\subsection{Existence of $T$-periodic solutions: a degree approach}\label{section-4.2}

In this section, using a topological degree approach firstly introduced in \cite{FeZa-15jde}
and then developed in several recent works (see \cite{BoFeZa-pp2015,BoFeZa-16,FeZa-15ade,FeZa-pp2015}), we prove the existence of a positive $T$-periodic solution of \eqref{eq-proof}. 

Since we are going to take advantage of the a-priori bound developed in Lemma~\ref{lem-4.1}, we assume again that
$f(s)$ is a convex function satisfying $(f_{1})$, $(f_{*})$ and $(f_{4})$; moreover, now also the superlinearity at zero condition $(f_{2})$ and the mean value
assumption $(a_{2})$ play a crucial role.

\begin{proposition}\label{prop-exist}
Let $a\colon \mathbb{R} \to \mathbb{R}$ be a $T$-periodic locally integrable function satisfying $(a_{1})$ and $(a_{2})$.
Let $f\in\mathcal{C}^{1}(\mathopen{[}0,\rho\mathclose{]})$ be a convex function satisfying
$(f_{1})$, $(f_{2})$, $(f_{*})$ and $(f_{4})$ with $M_{1}\in\mathopen{]}0,1\mathclose{[}$ and $M_{2}>0$ the constants given in Lemma~\ref{lem-4.1}. 
Then there exists a positive $T$-periodic solution $u^{*}(t)$ of \eqref{eq-proof} such that $\max_{t \in \mathbb{R}}u^{*}(t)<\rho$. 
\end{proposition}

\begin{proof}
We are going to use a topological argument based on \textit{Mawhin's coincidence degree theory} (cf.~\cite{Ma-93}).

First, taking into account condition $(f_{1})$, we introduce the $L^{1}$-Carath\'{e}odory function
\begin{equation*}
\tilde{f}(t,s) :=
\begin{cases}
\, -s, & \text{if } s \leq 0;\\
\, a(t)f(s), & \text{if } 0 \leq s \leq \rho;\\
\, a(t)f(\rho), & \text{if } s \geq \rho;
\end{cases}
\end{equation*}
and we consider the $T$-periodic problem associated with
\begin{equation}\label{eq-ftilde}
u'' + \tilde{f}(t,u) = 0.
\end{equation}
A standard maximum principle ensures that every $T$-periodic solution of \eqref{eq-ftilde} is non-negative;
moreover, in view of $(f_{2})$, if $u(t)$ is a $T$-periodic solution of \eqref{eq-ftilde} with $u\not\equiv0$, then $u(t)>0$ for all $t$.

Next, we write the $T$-periodic problem associated with \eqref{eq-ftilde} as a \textit{coincidence equation}
\begin{equation}\label{coinc-eq}
Lu = Nu,\quad u\in \text{\rm dom}\,L.
\end{equation}
As a first observation, let us recall that finding a $T$-periodic solution of \eqref{eq-ftilde} is equivalent to solving 
equation \eqref{eq-ftilde} on $\mathopen{[}0,T\mathclose{]}$ together with the periodic boundary condition
\begin{equation*}
u(0)=u(T), \qquad u'(0)=u'(T). 
\end{equation*}
Accordingly, let $X:=\mathcal{C}(\mathopen{[}0,T\mathclose{]})$ be the Banach space of continuous functions $u \colon \mathopen{[}0,T\mathclose{]} \to \mathbb{R}$,
endowed with the $\sup$-norm $\|u\|_{\infty} := \max_{t\in \mathopen{[}0,T\mathclose{]}} |u(t)|$,
and let $Z:=L^{1}(\mathopen{[}0,T\mathclose{]})$ be the Banach space of integrable functions $v \colon \mathopen{[}0,T\mathclose{]} \to \mathbb{R}$,
endowed with the norm $\|v\|_{L^{1}}:= \int_{0}^{T} |v(t)|~\!dt$.
Next we consider the differential operator
\begin{equation*}
L \colon u \mapsto - u'',
\end{equation*}
defined on
\begin{equation*}
\text{\rm dom}\,L := \bigl{\{}u\in W^{2,1}(\mathopen{[}0,T\mathclose{]}) \colon u(0) = u(T), \; u'(0) = u'(T) \bigr{\}} \subseteq X.
\end{equation*}
It is easy to prove that $L$ is a linear Fredholm map of index zero. Moreover, in order to enter the coincidence degree setting, we have to define the projectors
$P \colon X \to \ker L \cong {\mathbb{R}}$, $Q \colon Z \to \text{\rm coker}\,L \cong Z/\text{\rm Im}\,L \cong \mathbb{R}$,
the right inverse $K_{P} \colon \text{\rm Im}\,L \to \text{\rm dom}\,L \cap \ker P$ of $L$,
and the linear (orientation-preserving) isomorphism $J \colon \text{\rm coker}\,L \to \ker L$.
For the standard positions we refer to \cite[Section~2]{FeZa-15ade} and to \cite[Section~2]{BoFeZa-16}.
Finally, let us denote by $N \colon X \to Z$ the Nemytskii operator induced by the function $\tilde{f}(t,s)$, that is
\begin{equation*}
(N u)(t):= \tilde{f}(t,u(t)), \quad t\in \mathopen{[}0,T\mathclose{]}.
\end{equation*}

\smallskip

With this position, now we prove that there exists an open (bounded) set $\Omega\subseteq X$, with
\begin{equation}\label{eq-omega}
\Omega\subseteq B(0,\rho) \setminus \{0\},
\end{equation}
such that the coincidence degree $D_{L}(L-N,\Omega)$ of $L$ and $N$ in $\Omega$ is defined and different from zero.
In this manner, by the \textit{existence property} of the degree, there exists at least a nontrivial solution $u^{*}$ of \eqref{coinc-eq}
with $\|u^{*}\|_\infty < \rho$. Hence, $u^{*}(t)$ is a $T$-periodic solution of \eqref{eq-ftilde}.
As a consequence of the maximum principle, as already noticed, this solution is positive;
moreover, being $u^{*}(t)< \rho$ for any $t$, it solves the original equation \eqref{eq-proof}.

We split our argument into three steps. In the following, when referring to a solution $u(t)$ of \eqref{eq-step1} and \eqref{eq-step2}
we implicitly assume that $0\leq u(t) \leq \rho$, for all $t\in\mathbb{R}$, since $f(s)$ is defined on $\mathopen{[}0,\rho\mathclose{]}$.

\smallskip
\noindent
\textit{Step 1. }
\textit{There exists a constant $r\in\mathopen{]}0,\rho\mathclose{[}$ such that any $T$-periodic solution $u(t)$ of
\begin{equation}\label{eq-step1}
u'' + \vartheta a(t)f(u) = 0,
\end{equation}
for $0 < \vartheta \leq 1$, satisfies $\|u\|_{\infty} \neq r$.}

Indeed, since condition $(f_{2})$ can be written in the equivalent form
\begin{equation*}
\lim_{s\to 0^{+}}\dfrac{f(s)}{s}=0,
\end{equation*}
we can proceed (by contradiction) exactly as in the proof of \cite[Theorem~3.2]{FeZa-15ade} (thus, we omit the proof).

\smallskip
\noindent
\textit{Step 2. }
\textit{There exists a constant $\nu_{0} > 0$ such that any $T$-periodic solution $u(t)$ of
\begin{equation}\label{eq-step2}
u'' + a(t)f(u) + \nu \mathbbm{1}_{\bigcup_{i=1}^{m} I^{+}_{i}}(t)= 0,
\end{equation}
for $\nu \in \mathopen{[}0,\nu_{0}\mathclose{]}$, satisfies $\|u\|_{\infty} \neq \rho$. Moreover,
there are no $T$-periodic solutions $u(t)$ of \eqref{eq-step2} for $\nu = \nu_{0}$.} 

From Lemma~\ref{lem-4.1} we deduce that, for any $\nu\geq 0$, every $T$-periodic solution $u(t)$ of \eqref{eq-step2} satisfies $\|u\|_{\infty} \neq \rho$
(notice that the definition of
the extension $\tilde{f}(t,s)$ for $s\geq \rho$ has no role in this proof).
Next, we fix a constant $\nu_{0} > 0$ such that
\begin{equation*}
\nu_{0}>\dfrac{\|a\|_{L^{1}} \max_{0\leq s \leq \rho}f(s)}{\sum_{i=1}^{m} |I^{+}_{i}|}.
\end{equation*}
We have only to verify that, for $\nu = \nu_{0}$, there are no $T$-periodic solutions $u(t)$ of \eqref{eq-step2}.
Indeed, if $u(t)$ is a $T$-periodic solution of \eqref{eq-step2} 
then, integrating \eqref{eq-step2} on $\mathopen{[}0,T\mathclose{]}$, we obtain
\begin{equation*}
\nu \sum_{i=1}^{m} |I^{+}_{i}| = \nu \int_{0}^{T} \mathbbm{1}_{\bigcup_{i=1}^{m} I^{+}_{i}}(t)~\!dt \leq \int_{0}^{T}|a(t)|f(u(t))~\!dt \leq \|a\|_{L^{1}}\,\max_{0\leq s \leq \rho}f(s),
\end{equation*}
a contradiction with respect to the choice of $\nu_{0}$.

\smallskip
\noindent
\textit{Step 3. }
\textit{Computation of the degree. }
First of all, we compute the coincidence degree on $B(0,r)$. From \cite[Theorem~2.4]{Ma-93} and \textit{Step 1.}, we obtain that
\begin{equation}\label{eq-deg1}
D_{L}(L-N,B(0,r)) = \text{\rm deg}_{B}\biggl{(}-\dfrac{1}{T}\int_{0}^{T}\tilde{f}(t,\cdot)~\!dt,\mathopen{]}-r,r\mathclose{[},0\biggr{)} =1,
\end{equation}
where ``$\text{\rm deg}_{B}$'' denotes the classical Brouwer degree.
For the details, we refer to the proofs of \cite[Lemma~2.2]{BoFeZa-16} and \cite[Theorem~2.1]{FeZa-15ade}.

Secondly, we compute the coincidence degree on $B(0,\rho)$. From the \textit{homotopy invariance} of the degree and \textit{Step 2.}, we obtain that
\begin{equation}\label{eq-deg0}
D_{L}(L-N,B(0,\rho)) = 0.
\end{equation}
For the details, we refer to the proof of \cite[Theorem~2.1]{FeZa-15ade}.

In conclusion, from \eqref{eq-deg1}, \eqref{eq-deg0} and the \textit{additivity property} of the coincidence degree, we find that
\begin{equation*}
D_{L}(L-N, B(0,\rho) \setminus B[0,r]) = -1.
\end{equation*}
This ensures the existence of a nontrivial solution $u^{*}$ to \eqref{coinc-eq}
with 
\begin{equation*}
u^{*}\in \Omega := B(0,\rho) \setminus B[0,r].
\end{equation*}
Recalling \eqref{eq-omega} and the argument explained therein, the proof is concluded.
\end{proof}

\begin{remark}\label{rem-4.1}
It is worth noticing that the existence of a positive $T$-periodic solution to \eqref{eq-proof} could likely be proved under less restrictive assumptions
of $f(s)$. In particular, as shown in \cite{BoFeZa-16,FeZa-15ade}, the conclusion in \textit{Step 1.} is still valid when $f(s)$ is only continuous and
regularly oscillating at zero; on the other hand, again motivated by the results in the aforementioned papers we expect that \textit{Step 2.}
can be proved (with slightly different arguments) under alternative assumptions not requiring the convexity of $f(s)$.
We have chosen however to take advantage of the a-priori bound developed in Lemma~\ref{lem-4.1}, therefore giving the proof in this simplified setting,
since a convexity assumption will be in any case essential in the next Section~\ref{section-4.3}.    
$\hfill\lhd$
\end{remark}

\subsection{The Morse index computation}\label{section-4.3}

In this section we present the (crucial) Morse index computation.
As remarked in the introduction, it is based on an algebraic trick already employed in the proof of \cite[Theorem~1]{BrHe-90},
exploiting in an essential way the strict convexity assumption $(f_{3})$ (together with the sign condition $(f_{*})$).
Notice that all the other assumptions on $f(s)$ and $a(t)$ are not required in Lemma~\ref{lem-morse} below,
which is indeed an a-priori Morse index estimate, valid for positive $T$-periodic solutions of \eqref{eq-proof} independently of 
their existence. 

\begin{lemma}\label{lem-morse}
Let $a\colon \mathbb{R} \to \mathbb{R}$ be a $T$-periodic locally integrable function.
Let $f \in \mathcal{C}^{2}(\mathopen{[}0,\rho\mathclose{]})$ satisfy $(f_{*})$ and $(f_{3})$.
If $u(t)$ is a positive $T$-periodic solution of \eqref{eq-proof}, then
\begin{equation*}
\lambda_{0}\bigl{(}a(t)f'(u(t)) \bigr{)}<0,
\end{equation*}
where $\lambda_0\bigl{(}a(t)f'(u(t)) \bigr{)}$ denotes (as in Section~\ref{section-2}) the principal eigenvalue of the $T$-periodic problem associated with 
$v'' + (\lambda + a(t)f'(u(t))) v = 0$.
\end{lemma}

\begin{proof}
Let $u(t)$ be a positive $T$-periodic solution of \eqref{eq-proof}
and let $v(t)$ be a \textit{positive} eigenfunction associated to the principal eigenvalue $\lambda_{0}=\lambda_{0}(a(t)f'(u(t)))$.
Then, $v(t)$ satisfies
\begin{equation}\label{eq-4.10}
v''+ \bigl{(} \lambda_{0} + a(t)f'(u(t)) \bigr{)}v = 0,
\end{equation}
is $T$-periodic and $v(t)>0$ for all $t\in\mathbb{R}$ (cf.~\cite{CoLe-55}).

By multiplying \eqref{eq-proof} by $f'(u)v$ we obtain
\begin{equation*}
u''f'(u)v + a(t)f(u)f'(u)v = 0
\end{equation*}
and, respectively, by multiplying \eqref{eq-4.10} by $f(u)$ we have
\begin{equation*}
v''f(u) + \bigl{(} \lambda_{0} + a(t)f'(u) \bigr{)} v f(u) = 0.
\end{equation*}
From the above equalities, we therefore immediately deduce
\begin{align*}
\lambda_{0} v(t)f(u(t))
   &= -a(t)f(u(t))f'(u(t))v(t) - v''(t)f(u(t)) 
\\ &= u''(t)f'(u(t))v(t) - v''(t)f(u(t)), \quad \forall \,t\in\mathbb{R}.
\end{align*}
Integrating by parts this equality, we obtain
\begin{align*}
&\lambda_{0}\int_{0}^{T} v(t)f(u(t))~\!dt = \int_{0}^{T} \bigl{(}u''(t)f'(u(t))v(t) - v''(t)f(u(t)) \bigr{)}~\!dt
\\ &= \Bigl{[}-v'(t)f(u(t)) \Bigr{]}_{t=0}^{t=T} + \int_{0}^{T} \bigl{(}u''(t)f'(u(t))v(t) + v'(t)f'(u(t))u'(t) \bigr{)}~\!dt
\\ &= \int_{0}^{T} \bigl{(}u''(t)f'(u(t))v(t) + v'(t)f'(u(t))u'(t) \bigr{)}~\!dt.
\end{align*}
Via a further integration by parts, we find
\begin{align*}
&\int_{0}^{T} \bigl{(}v(t)f'(u(t))u''(t) + v'(t)f'(u(t))u'(t) \bigr{)}~\!dt
\\ &= \Bigl{[}v(t)f'(u(t)) u'(t)\Bigr{]}_{t=0}^{t=T} - \int_{0}^{T} v(t) f''(u(t))u'(t)^{2}~\!dt
\\ &= -\int_{0}^{T} v(t) f''(u(t))u'(t)^{2}~\!dt.
\end{align*}
In conclusion,
\begin{equation*}
\lambda_{0}\int_{0}^{T} v(t)f(u(t))~\!dt = -\int_{0}^{T} v(t) f''(u(t))u'(t)^{2}~\!dt.
\end{equation*}
Observing now that both the above integrals are positive, since $v(t)>0$ for all $t\in\mathbb{R}$ and $f(s)$ satisfies $(f_{3})$ and $(f_{*})$
(notice that $u'(t)\not\equiv0$, again in view of $(f_{*})$), we immediately deduce that
\begin{equation*}
\lambda_{0}=\lambda_{0}\bigl{(}a(t)f'(u(t)) \bigr{)}<0.
\end{equation*}
The lemma is thus proved.
\end{proof}

\begin{remark}\label{rem-indmorse1}
We observe that Lemma~\ref{lem-morse} in particular applies to the function $f(s) = s^{p}$ with $p > 1$, implying that,
whenever $u(t)$ is a positive $T$-periodic solution of $u'' + a(t) u^{p} = 0$, the Morse index of the linear equation
\begin{equation*}
v'' + p \, a(t) u(t)^{p-1}v = 0
\end{equation*}
is non-zero. On the other hand, it is worth noticing that
\begin{equation*}
\int_{0}^{T} a(t)  u(t)^{p-1}~\!dt < 0,
\end{equation*}
as it can be easily seen by writing $a(t)  u(t)^{p-1} = \tfrac{u''(t)}{p u(t)}$ and integrating by parts (compare with the computation leading to \eqref{meancon} in the introduction).
This provides an elegant proof of the claim made in Remark~\ref{confronto}: that is, for a linear equation $v'' + q(t)v = 0$, with $q(t)$ sign-changing,
the mean value condition $\int_{0}^{T} q(t)~\!dt \leq 0$ does not imply that the Morse index is zero. Also, this shows that the main result in \cite{BoZa-13} is not applicable to 
the equation $u'' + a(t) u^{p} = 0$, emphasizing the essential role of its abstract variant given in Proposition~\ref{propsub} (see again the discussion in Remark~\ref{confronto}).
$\hfill\lhd$
\end{remark}

\begin{remark}\label{rem-indmorse2}
Recalling that the Morse index of a positive $T$-periodic solution $u(t)$ of \eqref{eq-proof} is the Morse index of the linear equation
$v'' + a(t)f'(u(t)) v = 0$, Lemma~\ref{lem-morse} asserts that any positive $T$-periodic solution of \eqref{eq-proof} has non-zero Morse index.
From a variational point of view, this implies that $u(t)$, as a critical point of the action functional
\begin{equation*}
J(u) = \int_{0}^{T} \biggl{(} \frac{1}{2} u'(t)^{2} - a(t)F(u(t))\biggr{)} ~\!dt, \quad \text{where } F(u) = \int_{0}^{u} f(\xi)~\!d\xi,
\end{equation*}
is not a local minimum. We stress again that this is an a-priori information, valid for \textit{any} positive $T$-periodic solution of \eqref{eq-proof}; on the other hand,
it requires the global convexity assumption $(f_{3})$, which is usually not needed for existence results (see Remark~\ref{rem-4.1}).
It appears therefore a natural question if it is possible to prove, using variational arguments of mountain pass type (on the lines of \cite{AlTa-93,BeCaDoNi-95}),
the existence of \textit{at least one} positive $T$-periodic solution with non-zero Morse index, under less restrictive assumptions on $f(s)$.
This however does not seem to be an easy task, since - thought the local topological structure of a functional near a min-max critical point can be analyzed - estimates
from below for the Morse index are usually possible only under non-denegeracy assumptions (see, for instance, \cite{Gh-91,Ho-84}).
$\hfill\lhd$
\end{remark}

\subsection{Conclusion of the proof}\label{section-4.4}

We are now in a position to easily complete the proof of Theorem~\ref{th-proof}. Of course, we assume henceforth that $(f_{1})$, $(f_{2})$, $(f_{3})$ and 
$(f_{4})$ are satisfied, with $M_{1}, M_{2}$ the constants given by Lemma~\ref{lem-4.1}. As a consequence, $f(s)$ is (strictly) convex and the sign condition $(f_{*})$ holds true,
so that all the results in Section~\ref{section-4.1}, Section~\ref{section-4.2} and Section~\ref{section-4.3} can be used.

Let us define, for $(t,s) \in \mathbb{R}^{2}$, 
\begin{equation*}
h(t,s) :=
\begin{cases}
\, 0, & \text{if } s \leq 0;\\
\, a(t)\hat f(s), & \text{if } s \geq 0;
\end{cases}
\end{equation*}
where $\hat{f}(s)$ is given by \eqref{def-hat}. Using $(f_{1})$ and $(f_{2})$, it is easy to see that the function $h(t,s)$ satisfies
the smoothness conditions required in Proposition~\ref{propsub}.
Moreover, since $\hat{f}(s)$ has linear growth at infinity, the global continuability for the solutions of
\begin{equation}\label{eq-final}
u'' + h(t,u) = 0
\end{equation}
is guaranteed.
We now claim that both the assumptions $(i)$ and $(ii)$ of Proposition~\ref{propsub} are satisfied.

Indeed, Proposition~\ref{prop-exist} implies the existence of a $T$-periodic function $u^{*}(t)$, solving 
\eqref{eq-proof} and such that $0 < u^{*}(t) < \rho$ for any $t \in \mathbb{R}$; moreover, from Lemma~\ref{lem-morse} we know that
\begin{equation*}
\lambda_{0}\bigl{(}a(t)f'(u^{*}(t)) \bigr{)} < 0.
\end{equation*}
Since $h(t,s) = a(t)f(s)$ for $0 \leq s \leq \rho$, we have thus obtained a $T$-periodic solution of \eqref{eq-final} satisfying \eqref{hpmorse}.
Hence, condition $(i)$ is fulfilled.

As for condition $(ii)$, we simply take $\alpha(t) \equiv 0$, due to the fact that $\alpha(t)$ is a (trivial) solution of 
\eqref{eq-final}, and $0 = \alpha(t) < u^{*}(t)$ for any $t \in \mathbb{R}$.

Proposition~\ref{propsub} thus ensures that there exists $k^{*} \geq 1$ such that, for any integer $k \geq k^{*}$,
there exists an integer $m_{k} \geq 1$ such that, for any integer $j$ relatively prime with $k$ and such that
$1 \leq j \leq m_{k}$, equation \eqref{eq-final} has two subharmonic solutions $u_{k,j}^{(i)}(t)$ ($i=1,2$) of order $k$ (not belonging to the same periodicity class),
such that $u_{k,j}^{(i)}(t) - u^{*}(t)$ has exactly $2j$ zeros in the interval $\mathopen{[}0,kT\mathclose{[}$.
From the fact that \eqref{localiz} holds true (with $\alpha(t) \equiv 0$) together with Remark~\ref{remalpha}, we obtain that $u_{k,j}^{(i)}(t) > 0$ for any $t \in \mathbb{R}$.
We finally use Lemma~\ref{lem-4.1} (with $\nu = 0$) to ensure that $u_{k,j}^{(i)}(t) < \rho$ for any $t \in \mathbb{R}$. 
Thus $u_{k,j}^{(i)}(t)$ is a positive subharmonic solutions of equation \eqref{eq-proof} and the proof is concluded.
\qed

\begin{remark}\label{rem-4.4}
Reading more carefully the proof of Lemma~\ref{lem-morse}, one can notice that we do not use the fact that
the interval $\mathopen{[}0,\rho\mathclose{]}$ is a right neighborhood of zero. Indeed, the same conclusion holds true by
taking an interval $J\subseteq\mathbb{R}$ in place of $\mathopen{[}0,\rho\mathclose{]}$. Accordingly, we can state the following result.
\begin{quote}
\textit{Let $a\colon \mathbb{R} \to \mathbb{R}$ be a $T$-periodic locally integrable function. Let $J\subseteq\mathbb{R}$ be an interval.
Let $f \in \mathcal{C}^{2}(J)$ satisfy $f(s)>0$ and $f''(s)>0$ for any $s\in J$.
If $u(t)$ is a $T$-periodic solution of $u''+a(t)f(u)=0$ (thus, in particular, $u(t)\in J$ for any $t$), then
$\lambda_{0}\bigl{(}a(t)f'(u(t)) \bigr{)}<0$.}
\end{quote}

Although less general, Lemma~\ref{lem-morse} is the version more suitable to be subsequently applied to the search of positive subharmonics with range in $\mathopen{]}0,\rho\mathclose{[}$.
However, a natural question arises. Suppose to consider a nonlinearity $f(s)$ as above,
namely a $\mathcal{C}^{2}$-function which is positive and strictly convex in an interval $J\subseteq\mathbb{R}$.
Given a $T$-periodic solution $u^{*}(t)$ to $u''+a(t)f(u)=0$ (with $u^{*}(t)\in J$ for any $t$),
which additional conditions on $f(s)$ guarantee the applicability of the method adopted in this paper to find subharmonics of order $k$?
$\hfill\lhd$
\end{remark}

\section{Final remarks}\label{section-5}

We conclude the paper with a brief discussion about some natural questions which our result may suggest, if compared to the existing literature.  
As in the introduction, we focus our attention to the model superlinear equation
\begin{equation}\label{eq-fin}
u'' + a(t) u^{p} = 0,
\end{equation}
with $a(t)$ satisfying $(a_{1})$ and $(a_{2})$, and $p > 1$. 

Let us first recall that our main motivation for investigating the existence of positive subharmonic solutions to \eqref{eq-fin}
has been given by previous results obtaining positive subharmonic solutions when the weight function $a(t)$ has ``large'' negative part.  
More precisely, it was shown in \cite{BaBoVe-15,FeZa-pp2015} that, assuming to deal with a parameter-dependent weight
\begin{equation*}
a_\mu(t) := q^{+}(t) - \mu q^{-}(t),
\end{equation*}
equation \eqref{eq-fin} with $a(t) = a_{\mu}(t)$ has positive subharmonic solutions (of \textit{any} order)
whenever $\mu \gg 0$. Such a result, which may be interpreted in the context of singular perturbation problems, 
provides indeed positive subharmonic solutions 
which can be characterized by the fact of being either ``small'' or ``large'' on the intervals of positivity of the weight function 
(according to a chaotic-like multibump behavior).
A careful comparison between this result and Theorem~\ref{th-intro} could deserve some interest.

In a similar spirit, it is worth recalling that, again according to \cite{BaBoVe-15,FeZa-pp2015}, whenever $(a_{1})$ holds with $m \geq 2$,  
equation \eqref{eq-fin} with $a(t) = a_{\mu}(t)$ and $\mu \gg 0$ has at least $2^{m}-1$
distinct positive $T$-periodic solutions, say $u^{*}_{i}(t)$ for $i=1,\ldots,2^{m}-1$. Since 
Lemma~\ref{lem-morse} implies that any of these periodic solutions has non-zero Morse index,
Proposition~\ref{propsub} can be in principle applied $2^{m}-1$ times to obtain positive subharmonic solutions oscillating around
each $u^{*}_{i}(t)$. It seems however a quite delicate question to understand if these subharmonic solutions are actually distinct or not.  

Finally, we observe that it appears very natural to consider the damped
version of \eqref{eq-fin}, namely
\begin{equation}\label{eq-damp}
u'' + cu' + a(t) u^{p} = 0,
\end{equation}
where $c \in \mathbb{R}$ is an arbitrary constant. Indeed, it was shown in \cite{FeZa-15ade} that conditions
$(a_{1})$ and $(a_{2})$ also guarantee the existence of a positive $T$-periodic solution to \eqref{eq-damp}.
What about positive suhharmonic solutions? It is generally expected that the periodic solutions 
provided by the Poincar\'e-Birkhoff fixed point theorem disappear for (even small) perturbations destroying the Hamiltonian structure,
but maybe this is not the case for the positive subharmonic solutions to \eqref{eq-damp}.
Let us observe, for instance, that the multibump subharmonics constructed in \cite{FeZa-pp2015} via degree theory for
$a(t) = a_{\mu}(t)$ (and $\mu$ large) still exist for $c \neq 0$. Since both the symplectic approach and the variational one are useless
in a non-Hamiltonian setting, investigating the general case of an arbitrary weight function with a negative mean value seems to be a difficult problem.

\section*{Acknowledgements}\label{acknowledgements}

Work performed under the auspicies of the Grup\-po Na\-zio\-na\-le per l'Anali\-si Ma\-te\-ma\-ti\-ca, la Pro\-ba\-bi\-li\-t\`{a} e le lo\-ro
Appli\-ca\-zio\-ni (GNAMPA) of the Isti\-tu\-to Na\-zio\-na\-le di Al\-ta Ma\-te\-ma\-ti\-ca (INdAM).
Guglielmo Feltrin and Alberto Boscaggin are partially supported by the GNAMPA Project 2016
``Problemi differenziali non lineari: esistenza, molteplicit\`{a} e propriet\`{a} qualitative delle soluzioni''.
Alberto Boscaggin also acknowledges the support of the
project ERC Advanced Grant 2013 n.~339958 ``Complex Patterns for Strongly Interacting Dynamical Systems - COMPAT''.

\bibliographystyle{elsart-num-sort}
\bibliography{BF_biblio}

\bigskip
\begin{flushleft}

{\small{\it Preprint}}

{\small{\it May 2016}}

\end{flushleft}

\end{document}